%% file: closure_0606.tex
\documentclass[10pt,a4paper]{amsart} 
\usepackage{graphicx,latexsym,amssymb,color}
\theoremstyle{plain} 
\newtheorem{theorem}{Theorem}[section] 
\newtheorem{lemma}[theorem]{Lemma} 
 
\newtheorem{corollary}[theorem]{Corollary}

\theoremstyle{definition}
\newtheorem{definition}{Definition} 
\theoremstyle{remark} 
\newtheorem{remark}[theorem]{Remark} 
 
\newtheorem{claim}[theorem]{Claim} 
\newtheorem{convention}[theorem]{Convention}

\newtheorem*{acknowledgements}{Acknowledgements} 
\numberwithin{equation}{section}
\numberwithin{figure}{section}
\newcommand{\bd}{\begin{description}}   
\newcommand{\ed}{\end{description}} 
\newcommand{\ba}{\begin{array}}      \newcommand{\ea}{\end{array}} 
\newcommand{\bc}{\begin{center}}     \newcommand{\ec}{\end{center}} 
\newcommand{\be}{\begin{enumerate}}  \newcommand{\ee}{\end{enumerate}} 
\newcommand{\beq}{\begin{eqnarray}}  \newcommand{\eeq}{\end{eqnarray}} 
\newcommand{\beQ}{\begin{eqnarray*}} \newcommand{\eeQ}{\end{eqnarray*}} 
\newcommand{\bi}{\begin{itemize}}    \newcommand{\ei}{\end{itemize}}

\newcommand{\ov}{\overline} 
 
\newcommand{\ve}{\varepsilon} 
 
\newcommand{\la}{\lambda} 
\newcommand{\1}{\mathbf{1}}
\newcommand{\n}{ \{ 1,...,n \} }

\newcommand{\nbpt}{18}
\newcommand{\figtotext}[3]{\begin{array}{c}\includegraphics{#3}\end{array}}

\newcommand{\Over}{\figtotext{\nbpt}{\nbpt}{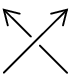}}
\newcommand{\under}{\figtotext{\nbpt}{\nbpt}{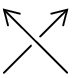}}
\newcommand{\smooth}{\figtotext{\nbpt}{\nbpt}{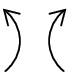}}
\newcommand{\si}{\sigma}
\begin{document} 
\title{Milnor invariants and the HOMFLYPT Polynomial} 

\author[J.B. Meilhan]{Jean-Baptiste Meilhan} 
\address{Institut Fourier, Universit\'e Grenoble 1 \\
         100 rue des Maths - BP 74\\
         38402 St Martin d'H\`eres , France}
	 \email{jean-baptiste.meilhan@ujf-grenoble.fr}
\author[A. Yasuhara]{Akira Yasuhara} 
\address{Tokyo Gakugei University\\
         Department of Mathematics\\
         Koganeishi \\
         Tokyo 184-8501, Japan}
	 \email{yasuhara@u-gakugei.ac.jp}

\thanks{The first author is partially supported by 
the JSPS Invitation Fellowship for Research in Japan (Short Term) ($\#$S-10127). 
The second author is partially supported by a JSPS Grant-in-Aid for Scientific Research (C) 
($\#$20540065). }  

%
%
%
\begin{abstract} 
We give formulas expressing Milnor invariants of an $n$-component link $L$ in the 3-sphere in terms of the HOMFLYPT polynomial as follows. 
If the Milnor invariant $\ov{\mu}_J(L)$ vanishes for any sequence $J$ with length at most $k$, 
then any Milnor $\ov{\mu}$-invariant $\ov{\mu}_I(L)$ with length between $3$ and $2k+1$ 
can be represented as a combination of HOMFLYPT polynomial of knots obtained from the link by certain band sum operations.
In particular, the \lq first non vanishing' Milnor invariants can be always represented as such a linear combination. 
\end{abstract} 
\maketitle
\section{Introduction}\label{sec:intro}
J. Milnor defined in \cite{Milnor,Milnor2} a family of link invariants, known as \emph{Milnor $\ov{\mu}$-invariants}. 
Here, and throughout the paper, by a link we mean an oriented, ordered link in $S^3$.  
Roughly speaking, Milnor invariants encode the behaviour of parallel copies of each link component 
in the lower central series of the link group.  
Given an $n$-component link $L$ in $S^3$, Milnor invariants are specified by a sequence $I$ of (possibly repeating) indices from $\n$. 
The length of the sequence is called the \emph{length} of the Milnor invariant. 
It is known that Milnor invariants of length two are just linking numbers. 
However in general, Milnor invariant $\ov{\mu}_L(I)$ is only well-defined 
modulo the greatest common divisor $\Delta_L(I)$ of all Milnor invariants $\ov{\mu}_L(J)$ such that $J$ is obtained from 
$I$ by removing at least one index and permuting the remaining indices cyclicly.  
This indeterminacy comes from the choice of the meridian curves generating the link group.  
Equivalently, it comes from the indeterminacy of representing the link as the closure of a string link \cite{HL}. 
See Section \ref{sec:milnor} for the definitions. 

Recall that the HOMFLYPT polynomial of a knot $K$ is of the form $P(K;t,z)=\sum_{k=0}^N P_{2k}(K;t)z^{2k}$, 
and denote by $P_{0}^{(l)}(K)$ the $l$th derivative of $P_{0}(K;t)\in {\Bbb Z}[t^{\pm 1}]$ evaluated at $t=1$.  
Let us denote by $(\log P_0(K))^{(n)}$ the $n$th derivative of $\log P_0(K;t)$ evaluated at $t=1$.  
Since $P_0(K;1)=1$, we have 
 $$ (\log P_0 )^{(n)}=  P^{(n)}_0 + \sum_{k_1+...+k_{m}=n} n_{(k_1,...,k_m)} P^{(k_1)}_0...P^{(k_m)}_0, $$ 
where the sum runs over all $k_1,...,k_m$ such that $k_1+\cdots+k_{m}=n~(k_i\geq 2)$, and where $n_{(k_1,...,k_m)}\in \mathbb{Z}$.  
For example, one can check that $(\log P_0)^{(n)}=P^{(n)}_0$ for $n=1,2,3$, and that 
$(\log P_0)^{(4)}=P^{(4)}_0 - 3 (P^{(2)}_0)^2$.  

In this paper, we show the following.  
If Milnor invariant $\ov{\mu}_J(L)$ vanishes for any sequence $J$ with length at most $k$, 
then any Milnor $\ov{\mu}$-invariant $\ov{\mu}_I(L)$ with length $m+1~(3\leq m+1\leq 2k+1)$ is given by a 
linear combination 
of $(\log P_{0})^{(m)}$ invariants of knots obtained from the link by certain band sum operations. 

For simplicity, we first state the formula for Milnor link-homotopy invariants $\ov{\mu}(I)$, 
\emph{i.e.} such that the sequence $I$ has no repeated index.  
Let $L=\bigcup_{i=1}^n L_i$ be an 
$n$-component link in $S^3$.
Let $I=i_1i_2...i_{m}$ be a sequence of $m$ distinct elements of $\{1,...,n\}$.  
Let $B_I$ be an oriented $(2m)$-gon, and denote by $p_j~(j=1,...,m)$ a set of $m$ non-adjacent edges 
of $B_I$ according to the boundary orientation. 
Suppose that $B_I$ is embedded in $S^3$ such that $B_I\cap L=\bigcup_{j=1}^{m} p_j$, and 
such that each $p_j$ is contained in $L_{i_j}$ 
with opposite orientation. 
We call such a disk an \emph{I-fusion disk} for $L$.  
For any subsequence $J$ of $I$, 
we define the oriented knot $L_J$ as the closure of  
$\left( (\bigcup_{i\in \{J\}} L_i)\cup \partial B_I \right)\setminus \left( (\bigcup_{i\in \{J\}} L_i)\cap B_I \right)$,  
where $\{J\}$ is the subset of $\n$ formed by all indices appearing in the sequence $J$. 

\begin{theorem}\label{thm:main}
Let $L$ be an $n$-component link in $S^3$ ($n\ge 3$) with vanishing 
Milnor link-homotopy invariants of length up to $k$.   
Then for any sequence $I$ of $(m+1)$ distinct elements of $\{1,...,n\}$ $(3\leq m+1\leq 2k+1)$ and 
for any $I$-fusion disk for $L$, we have 
\[ \overline{\mu}_L(I) \equiv \frac{-1}{m!2^{m}}\sum_{J<I}(-1)^{|J|} (\log P_0(L_J))^{(m)} \textrm{ (mod $\Delta_L(I)$)}, \]
where the sum runs over all subsequences $J$ of $I$, and where $|J|$ denotes the length of the sequence $J$.
\end{theorem}
This generalizes widely a result of M. Polyak for Milnor's triple linking number $\ov{\mu}(123)$ \cite{P}. 
There are several other known results relating Milnor invariants of (string) links 
to the Alexander polynomial, for example see \cite{C,L,MV,M,Mu,T} 
(note in particular that the results of \cite{L,MV,M,P} make use of closure-type operations). 
Relation to quantum invariants are also known via the Kontsevich integral \cite{HM}. 

We emphasize that our assumption that the link has vanishing Milnor link-homotopy invariants of length up to $k$ is 
essential in order to compute its Milnor invariants of length up to $2k+1$ using our formula. 
See the example at the end of this paper, which shows that Milnor link-homotopy invariants of length 
4 are not given by the formula in Theorem~\ref{thm:main} if there are nonvanishing linking numbers. 

As noted above, only the first non vanishing Milnor invariants of a link are well-defined integer-valued invariants.  
The following is in some sense a refinement of Theorem \ref{thm:main} for the first non vanishing Milnor invariants. 

\begin{theorem}\label{thm:main2}
Let $L$ be an $n$-component link in $S^3$ ($n\ge 3$) with vanishing 
Milnor link-homotopy invariants of length up to $k(\geq 2)$.  
Then for any sequence $I$ of $(k+1)$ distinct elements of $\{1,...,n\}$ and 
for any $I$-fusion disk for $L$, we have 
\[ \overline{\mu}_L(I) = \frac{-1}{k!2^{k}}\sum_{J<I}(-1)^{|J|} P_0^{(k)}(L_J)\in{\mathbb Z}. \]
\end{theorem}
Theorem~\ref{thm:main2} implies that all Milnor link-homotopy invariants of a link $L$ vanish if and only if 
all linking numbers of $L$ vanish and 
$\sum_{J<I}(-1)^{|J|} P_0^{(k)}(L_J)=0$
for all $k~(2\leq k\leq n-1)$ and for all nonrepeated sequences $I$ of length $k+1$.

Theorems~\ref{thm:main} and \ref{thm:main2} generalize as follows. 
Let $L=\bigcup_{i=1}^n L_i$ be an $n$-component link in $S^3$, and 
let $I=i_1 i_2 \cdots i_m$ be a sequence of $m$ elements of $\{1,...,n\}$, where each 
element $i$ appears exactly $r_i$ times.  
Denote by $D_I(L)$ the $m$-component link obtained from $L$ as follows. 
\begin{itemize}
\item Replace each string $L_i$ by $r_i$ zero-framed parallel copies of it, labeled 
from $L_{(i,1)}$ to $L_{(i,r_i)}$.  
If $r_i=0$ for some index $i$, simply delete $L_i$.  
\item 
Let $D_I(L)=L'_1\cup\cdots\cup L'_m$ be the $m$-string link $\bigcup_{i,j} L_{(i,j)}$ 
with the order induced by the lexicographic order of the index $(i,j)$.  
This ordering defines a bijection $\varphi:\{(i,j)~|~1\leq i\leq n, 1\leq j\leq r_i\}\rightarrow\{1,...,m\}$.   
\end{itemize}
We also define a sequence $D(I)$ of elements of $\{1,...,m\}$ without repeated index as follows.
First, consider a sequence of elements of $\{ (i,j) ; 1\leq i\leq n, 1\leq j\leq r_i \}$ by replacing each $i$ in $I$ with $(i,1),...,(i,r_i)$
in this order.  For example if $I=12231$, we obtain the sequence $(1,1),(2,1),(2,2),(3,1),(1,2)$.  Next replace each term $(i,j)$ of this sequence with $\varphi((i,j))$.  
Hence we have $D(12231)=13452$.  

\begin{theorem}\label{thm:general}
Let $L$ be an $n$-component link in $S^3$ with vanishing Milnor invariants of length up to $k$.   
Let $I$ be a sequence of $(m+1)$ possibly repeating elements of $\{1,...,n\}$ $(3\leq m+1\leq 2k+1)$.   
\begin{enumerate}
 \item[(i)] For any $D(I)$-fusion disk for $D_I(L)$, we have 
\[ \ov{\mu}_L(I) \equiv \frac{-1}{m!2^{m}}\sum_{J< D(I)}(-1)^{|J|}(\log  P_0(D_I(L)_J))^{(m)} \textrm{ (mod $\Delta_L(I)$)}. \]
 \item[(ii)] If $m=k(\geq 2)$, then for any $D(I)$-fusion disk for $D_I(L)$, we have 
\[ \ov{\mu}_L(I) = \frac{-1}{k!2^{k}}\sum_{J< D(I)}(-1)^{|J|}P_0^{(k)}(D_I(L)_J) \in \mathbb{Z}. \]
\end{enumerate}
\end{theorem}
Theorem~\ref{thm:general} follows directly from Theorems~\ref{thm:main} and \ref{thm:main2}, 
since Milnor proved in \cite[Thm. 7]{Milnor2} that 
 $\ov{\mu}_{D_I(L)}(D(I))=\ov{\mu}_{L}(I)$  
(note that $\Delta_L(I)=\Delta_{D_I(L)}(D(I))$, again as a consequence of \cite[Thm. 7]{Milnor2}).

This implies that all Milnor invariants of a link $L$ vanish if and only if 
all linking numbers of $L$ are zero and 
$\sum_{J< D(I)}(-1)^{|J|}P_0^{(k)}(D_I(L)_J)=0$
for all $k(\geq 2)$ and for all sequences $I$ with length $k+1$.

The rest of the paper is organized as follows. 
In Section 2, we review some elements of the theory of claspers, which is the main tool in proving our main results.
In Section 3, we recall some properties of the HOMFLYPT polynomial of knots.  
In Section 4, we review Milnor invariants and string links, and give a few lemmas. 
Section 5 is devoted to the proof of Theorem~\ref{thm:main}. 
In Section 6, we prove Theorem~\ref{thm:main2} and show, as a consequence, how to use the HOMFLYPT polynomial 
to distinguish string links up to link-homotopy. 
The paper is concluded by a simple example which illustrates the necessity of the assumptions required in our results. 

\begin{convention}
In this paper, given a sequence $I$ of elements of $\{1,...,n\}$, the notation 
$J<I$ will be used for any subsequence $J$ of $I$, possibly empty or equal to $I$ itself. 
By $J\lneq I$, we mean any subsequence $J$ of $I$ that is not $I$ itself. 
We will use the notation $\{I\}$ for the subset of $\n$ formed by all indices appearing in the sequence $I$, and 
$|I|$ will denote the length of the sequence $I$.   
\end{convention}

\begin{acknowledgements}
The authors thank 
Kazuo Habiro for useful discussions.
\end{acknowledgements}

\section{Some elements of clasper theory} \label{sec:claspers}
The primary tool in the proofs of our results is the theory of claspers.  
We recall here the main definitions and properties of this theory
that will be useful in subsequent sections.  
For a general definition of claspers, we refer the reader to \cite{H}.  

\begin{definition}\label{defclasp}
Let $L$ be a (string) link.  
A surface $G$ embedded in $S^3$ (or $D^2\times [0,1]$) is called a {\em graph clasper} for $L$ if it satisfies the following three conditions:
\be 
\item $G$ is decomposed into disks and bands, called {\em edges}, each of which connects two distinct disks.
\item The disks have either 1 or 3 incident edges, and are called {\em leaves} or {\em vertices} respectively.
\item $G$ intersects $L$ transversely, and the intersections are contained in the union of the interiors of the leaves. 
\ee
In particular, if a graph clasper $G$ is a disk, we call it a \emph{tree clasper}.  
\end{definition}

Throughout this paper, the drawing convention for claspers are those of \cite[Fig.~7]{H}, unless otherwise specified.  

The degree of a connected graph clasper $G$ is defined as half of the number of vertices and leaves.  
A tree clasper of degree $k$ is called a \emph{$C_k$-tree}.  
Note that a $C_k$-tree has exactly $(k+1)$ leaves.

A graph clasper for a (string) link $L$ is \emph{simple} if each of its leaves intersects $L$ at exactly one point.   

Let $G$ be a simple graph clasper for an $n$-component (string) link $L$.  
The \emph{index} of $G$ is the collection of all integers $i$ such that $G$ intersects the $i$th component of $L$. 
For example, if $G$ intersects component $3$ twice and components $2$ and $5$ once, and is disjoint from all other components of $L$, then its index is $\{2,3,5\}$.

Given a graph clasper $G$ for a (string) link, there is a procedure to construct a framed link, in a regular neighbourhood of $G$.  There is thus a notion of \emph{surgery along $G$}, which is defined as surgery along the corresponding framed link.  
In particular, surgery along a simple $C_k$-tree is a local move as illustrated in Figure \ref{ckmove}. 
\begin{figure}[!h]
\includegraphics{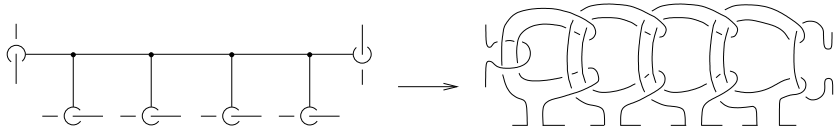}
\caption{Surgery along a simple $C_5$-tree.} \label{ckmove}
\end{figure}

The $C_k$-equivalence is the equivalence relation on (string) links generated by surgeries 
along connected graph claspers of degree $k$ and isotopies.  
Alternatively, the $C_k$-equivalence can be defined in term of ``insertion'' of elements of the 
$k$th term of the lower central series of the pure braid group \cite{stanford}.
We use the notation $L \stackrel{C_{k}}{\sim} L'$ for $C_k$-equivalent (string) links $L$ and $L'$.  

It is known that the $C_k$-equivalence becomes finer as $k$ increases, and that 
two links are $C_k$-equivalent if and only if they are related by surgery along simple $C_k$-trees \cite{H}.
Moreover, it was shown by Goussarov and Habiro that this equivalence relation characterizes the topological information 
carried by finite type knot invariants. More precisely, it is shown in \cite{G,H} that two knots are 
$C_k$-equivalent if and only if they cannot be distinguished by any finite type invariant of degree $<k$.  

\subsection{Linear trees and planarity}

For $k\ge 3$, a $C_k$-tree $G$ having the shape of the tree clasper in Figure \ref{ckmove} is called a \emph{linear} $C_k$-tree. 
The left-most and right-most leaves of $G$ in Figure \ref{ckmove} are called the \emph{ends} of $G$.

Now suppose that the $G$ is a linear $C_k$-tree for some knot $K$, 
and denote its ends by $f$ and $f'$.  
Then the remaining $(k-1)$ leaves of $G$ can be labelled from $1$ to $(k-1)$, by 
travelling along the boundary of the disk\footnote{Recall that a clasper is an embedded surface: in particular, 
since $T$ is a tree clasper, the underlying surface is homeomorphic to a disk.  } $G$ from $f$ to $f'$ so that all leaves are visited.  
We say that $G$ is \emph{planar} if, when travelling along $K$ from $f$ to $f'$, either following or against the orientation, the labels of the leaves met successively are strictly increasing.  
   \begin{figure}[!h]
    \includegraphics{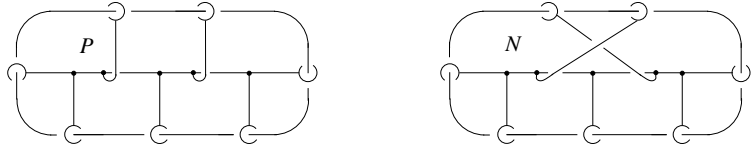}
   \caption{The $C_6$-tree $P$ is planar for the unknot, while $N$ is nonplanar.  }\label{fig:planar}
   \end{figure}
See Figure \ref{fig:planar} for an example.  
\subsection{Calculus of Claspers for parallel claspers}
We shall need refinements of Propositions 4.4 and 4.6 of \cite{H} for parallel tree claspers. 
 
Here, by \emph{parallel} tree claspers we mean 
a family of $m$ parallel copies of a tree clasper $T$, for some $m\ge 1$.  
We call $m$ the \emph{multiplicity} of the parallel clasper.    
Note that there is no ambiguity in the notion of parallel copies here, since for a tree clasper the underlying surface is homeomorphic to a disk. 

\begin{lemma}\label{lem:parallel}
Let $m, k,k'\ge 1$ be integers. 
Let $T$ be a parallel $C_k$-tree with multiplicity $m$ for a (string) link $L$, and let $T'$ be a $C_{k'}$-tree for $L$, disjoint from $T$. 
\begin{enumerate}
 \item Let $\tilde{T}\cup \tilde{T'}$ be obtained from $T\cup T'$ by sliding a leaf $f'$ of $T'$ over $m$ parallel leaves of $T$ (see Figure~\ref{fig:slide}(a)). 
   Then $L_{T\cup T'}$ is ambient isotopic to $L_{\tilde{T}\cup \tilde{T'}\cup Y\cup C}$, where 
   $Y$ denotes the parallel $C_{k+k'}$-tree with multiplicity $m$ obtained by inserting a vertex $v$ in the edge $e$ of $T$  
   and connecting $v$ to the edge incident to $f'$ as shown in Figure \ref{fig:slide}~(a), 
   and where $C$ is a disjoint union of $C_{k+k'+1}$-trees for $L$. 
 \item Let $\tilde{T}\cup \tilde{T'}$ be obtained from $T\cup T'$ by passing an edge of $T'$ across $m$ parallel edges of $T$ (see Figure~\ref{fig:slide}(b)).  
   Then $L_{T\cup T'}$ is ambient isotopic to $L_{\tilde{T}\cup \tilde{T'}\cup H\cup C}$, where 
   $H$ denotes the parallel $C_{k+k'+1}$-tree with multiplicity $m$ obtained by 
   inserting a vertex $v$ in both edges, and connecting them by an edge 
   as shown in Figure~\ref{fig:slide}(b)), and where $C$ is a disjoint union of $C_{k+k'+2}$-trees for $L$. 
 \end{enumerate}
   \begin{figure}[!h]
    \includegraphics{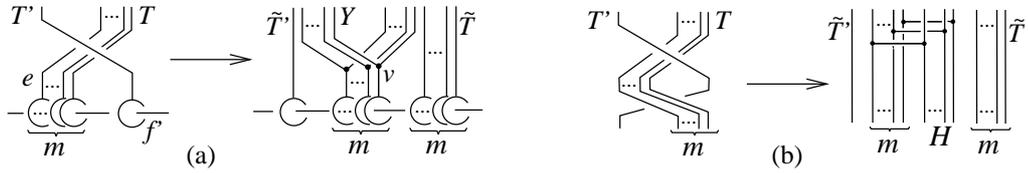}
   \caption{Leaf slide and crossing change involving parallel tree claspers.  }\label{fig:slide}
   \end{figure}
\end{lemma}
\noindent 
This result is well-known for $m=1$. 
The general case is easily proved using the arguments of the proof of \cite[Prop. 4.4]{H} and \cite[Prop. 4.6]{H} respectively.  

\begin{remark}\label{rem:parallel}
Notice that, following the proofs of \cite[Prop. 4.4 and 4.6]{H}, the index of each of the tree claspers involved in Lemma 
\ref{lem:parallel} can be determined from those of $T$ and $T'$ as follows. 
We have that the index of $\tilde{T}$ is equal to the index of $T$, the index of $\tilde{T}'$ is equal to the index of $T'$, and 
the indices of $Y$, $H$ and each connected component of $C$ are equal to the union of the indices of $T$ and $T'$. 
\end{remark}
\section{The HOMFLYPT polynomial}\label{sec:homfly}
In this section, we recall the definition of the HOMFLYPT polynomial, and mention a few useful examples and properties. 

The \emph{HOMFLYPT polynomial} $P(L;t,z)\in {\Bbb Z}[t^{\pm 1},z^{\pm 1}]$ of an oriented link $L$ is defined by the following formulas
\begin{enumerate}
\item $P(U;t,z) = 1$, 
\item $t^{-1}P(L_+;t,z) - tP(L_- ;t,z) = zP(L_0 ;t,z)$, 
\end{enumerate}
where $U$ denotes the unknot and where $L_+$, $L_-$ and $L_0$ are three links that are identical except in a $3$-ball where they look as follows: 
$$L_+=\Over\quad ;\quad L_-=\under\quad ; \quad L_0=\smooth.$$
In particular, the HOMFLYPT polynomial of an $r$-component link $K$ is of the form  
 $$ P(K;t,z)=\sum_{k=1}^N P_{2k-1-r}(K;t)z^{2k-1-r}, $$
where $P_{2k-1-r}(K;t)\in {\Bbb Z}[t^{\pm 1}]$ is called the $(2k-1-r)$th coefficient polynomial of $K$.  
Furthermore, the lowest degree coefficient polynomial of $K$ is given by 
\begin{equation}\label{eq:LM}
 P_{1-r}(K;t)=t^{2Lk(L)}(t^{-1}-t)^{r-1}\prod_{i=1}^r P_0(K_i;t), 
\end{equation}
where $K_i$ denotes the $i$th component of $K$, and where $Lk(L):=\sum_{i<j} lk(L_i,L_j)$, see \cite[Prop. 22]{LM}.

Denote by $P_{k}^{(l)}(L)$ the $l$th derivative of $P_{k}(L;t)$ evaluated at $t=1$.  
It was proved by Kanenobu and Miyazawa that $P_{k}^{(l)}$ is a finite type invariant of degree $k+l$ \cite{KM}.  
In particular, $P_{0}^{(l)}$ is of degree $l$, and thus is an invariant of $C_{l+1}$-equivalence. 

It is well-known that the HOMFLYPT polynomial of knots is multiplicative under connected sum.  
Thus the same holds for the lowest degree coefficient polynomial $P_0$, and in general, for any integer $n$ and any two oriented knots $K$ and $K'$,
we have 
 $$ P_0^{(n)}(K\sharp K') = 
    P_0^{(n)}(K) + P_0^{(n)}(K') + \sum_{k=1}^{n-1} \binom{n}{k} P_0^{(k)}(K)P_0^{(n-k)}(K'). $$
If, moreover, we assume that the knot $K$ is $C_n$-equivalent to the unknot, then we have 
\begin{equation}\label{eq:nadditivity}
 P_0^{(n)}(K\sharp K') = P_0^{(n)}(K) + P_0^{(n)}(K'),
\end{equation}
since $P_0^{(n)}$ is an invariant of $C_{n+1}$-equivalence, for all $k$.   

In general, a simple way to derive an additive knot invariant from the coefficient polynomial $P_0$ is to take its log. 
(Since $P_0(K;t)$ is in ${\Bbb Z}[t^{\pm1}]$ and $P_0(K;1)=1$ for any knot $K$, 
$\log P_0(K;t)$ can be regarded as a smooth function defined on an open interval which contains $1$.)   
Indeed, we have that, for any two oriented knots $K$ and $K'$, 
 $$ (\log P_0)(K\sharp K';t) = (\log P_0)(K;t) + (\log P_0)(K';t). $$
\noindent 

Denote by $(\log P_0(K))^{(n)}$ the $n$th derivative of $\log P_0(K;t)$ evaluated at $t=1$.  
As mentioned in the introduction, $(\log P_0(K))^{(n)}$ is equal to $P_0(K)^{(n)}$ plus a sum of products of $P_0(K)^{(k)}$'s with $k<n$.  
So we see that $(\log P_0)^{(n)}$ is an additive finite type knot invariant of degree $n$, and thus is an invariant of $C_{n+1}$-equivalence.

The following simple example shall be useful later.
\begin{lemma}\label{lem:homfly_planar}
Let $n\ge 1$, and let $\varepsilon=(\varepsilon_0,\varepsilon_1,...,\varepsilon_n,\varepsilon_{n+1})\in \{-1,1\}^{n+2}$. 
Let $K^{\varepsilon}_n$ be the knot represented in Figure \ref{fig:Kn}.  Then 
  $$ \big(\log P_0( K^{\varepsilon}_n )\big)^{(n+1)} = P_0^{(n+1)}( K^{\varepsilon}_n ) = (-1)^n 2^{n+1}(n+1)! \prod_{i=0}^{n+1}\varepsilon_i. $$
\end{lemma}
\noindent Notice that $K^{\varepsilon}_n$ is $C_{n+1}$-equivalent to the unknot, and that for all $k\le n$, 
we thus have 
$\big(\log P_0( K^{\varepsilon}_n )\big)^{(k)} = P_0^{(k)}( K^{\varepsilon}_n ) = 0$.  
\begin{proof}
Let us prove the second equality.
We first prove the formula for the knot $K^+_n:=K_n^{(1,...,1)}$, by induction on $n$.  
Since $K^{+}_{1}$ is the trefoil, we have $P^{(2)}_0( K^{+}_{1})=-8$. 
Suppose that $n>1$. 
Clearly, changing the crossing $c$ of $K^+_n$ yields the unknot 
(see Figure \ref{fig:Kn}).  
Hence
 $$  P_0( K^+_n;t ) = t^2 + tP_{-1}( L_n;t ), $$
where $L_n=K_1\cup K_2$ is the $2$-component link represented on the right-hand side of Figure \ref{fig:Kn}.  
 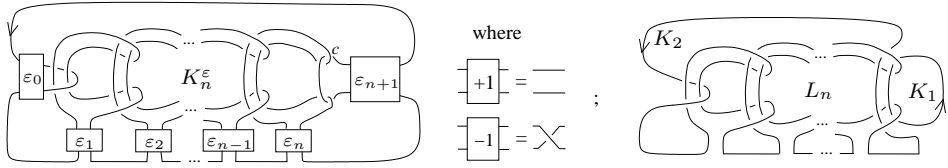
\begin{figure}[!h]
  \input{kn.pstex_t}
  \caption{The knot $K^{\varepsilon}_n$ and the $2$-component link $L_n$. }\label{fig:Kn}
 \end{figure}
Notice that $K_1$ is an unknot, while $K_2$ is a copy of the knot $K^{+}_{n-1}$. Hence by Equation (\ref{eq:LM}) we have 
$$
 P_0( K^{+}_n;t ) = t^2 + (1-t^2) P_0( K^{+}_{n-1};t ). 
$$
By differentiating this equation $(n+1)$ times and evaluating at $t=1$, we obtain
\[P^{(n+1)}_0( K^{+}_n) = -2(n+1)P^{(n)}_0( K^{+}_{n-1}) - n(n+1)P^{(n-1)}_0( K^{+}_{n-1}). \]
Since $K^{+}_{n-1}$ is $C_n$-equivalent to the unknot, we see that $P_0^{(n-1)}(K^{+}_{n-1})=0$. 
Hence we have $P^{(n+1)}_0( K^{+}_n) = -2(n+1)P^{(n)}_0( K^{+}_{n-1})$. 
The induction hypothesis implies $P^{(n+1)}_0( K^{+}_n) = (-1)^n2^{n+1}(n+1)!$.

Now, notice that in general $K^{\varepsilon}_n$ is obtained from the unknot by surgery along the linear $C_{n+1}$-tree represented 
in Figure \ref{fig:knclasper}. 
 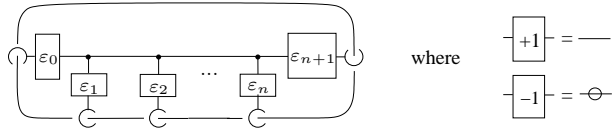
\begin{figure}[!h]
  \input{knclasper.pstex_t}
  \caption{Here, a $\ominus$ on an edge represents a negative half-twist. }\label{fig:knclasper}
 \end{figure}
It follows from \cite[Claim in p-36]{H} that
$K^{\varepsilon}_n \# K^+_n$ is $C_{n+2}$-equivalent to the unknot (resp. to $K^+_n\# K^+_n$) 
if $\prod_{i=0}^{n+1}\varepsilon_i$ is equal to $-1$ (resp. to $=1$).
Since $P^{(n+1)}_0$ is an invariant of $C_{n+2}$-equivalence and 
the knot $K^{\varepsilon}_n$ is $C_{n+1}$-equivalent to the unknot for any $\varepsilon\in \{-1,1\}^{n+2}$, we have
$$(1+\prod_{i=0}^{n+1}\varepsilon_i)P^{(n+1)}_0(K^+_n)=P^{(n+1)}_0(K^{\varepsilon}_n\# K^+_n)=P^{(n+1)}_0(K^{\varepsilon}_n)+P^{(n+1)}_0(K^+_n)$$
Hence we have
$$ P^{(n+1)}_0(K^{\varepsilon}_n) =  P^{(n+1)}_0(K^{+}_n)\prod_{i=0}^{n+1}\varepsilon_i. $$  
The second equality follows.  

Recall that $(\log P_0)^{(n+1)}$ is given by the sum of $P_0^{(n+1)}$ 
and a combination of $P_0^{(k)}$'s with $k\le n$.   
Since the knot $K^{\varepsilon}_n$ is $C_{n+1}$-equivalent to the unknot, 
the first equality follows. 
\end{proof}

We note that, according to the above proof, Lemma \ref{lem:homfly_planar} gives the variation of 
$\big(\log P_0\big)^{(n+1)}$ and $P_0^{(n+1)}$
under surgery along a planar linear $C_{n+1}$-tree for the unknot.  
On the other hand, the HOMFLYPT polynomial does not change under surgery along a non-planar tree clasper, 
as follows from a formula of Kanenobu \cite{Kanenobu}.  
\begin{lemma}\label{lem:Pnonplanar}
Let $T$ be a non-planar linear tree clasper for a knot $K$.  Then 
$P_0(K_T;t) = P_0(K;t)$.
\end{lemma}
\begin{proof}
We may assume that $K_T$ and $K$ are given by identical diagrams, except in a disk 
where they differ as illustrated in Figure \ref{fig:Hnonplanar}.
 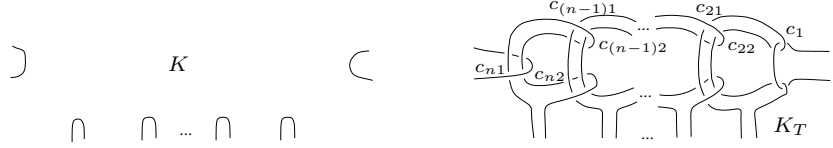
\begin{figure}[!h]
  \input{kane.pstex_t}
  \caption{The knots $K$ and $K_T$. }\label{fig:Hnonplanar}
 \end{figure}
Let $L[\varepsilon_2,...,\varepsilon_n]~(\varepsilon_j\in \{-1,1\},~j=2,...,n)$ be the link
obtained from $K_T$ by smoothing the crossing $c_1$, and
\begin{itemize}
 \item[(i)] smoothing the crossing $c_{j1}$ if $\varepsilon_j=1$, or
 \item[(ii)] changing the crossing $c_{j1}$ and smoothing the crossing $c_{j2}$ if $\varepsilon_j=-1$. 
\end{itemize}
Kanenobu showed that if all links $L[\varepsilon_2,...,\varepsilon_n]$ have less than
$n+1$ components, then $P_0(K_T;t) = P_0(K;t)$, see \cite[(3.9)]{Kanenobu}. 
Moreover, Kanenobu showed how to estimate the number of components of $L[\varepsilon_2,...,\varepsilon_n]$ 
using a kind of a chord diagram which corresponds to the smoothed crossings.
More precisely, the chord diagram associated to $L[\varepsilon_2,...,\varepsilon_n]$ 
represents the $n$-singular knot obtained from $K_T$ by changing the crossing $c_1$ 
and each crossing $c_{j1}$ (resp. $c_{j2}$) such that $\varepsilon_j=1$ (resp. $\varepsilon_j=-1$) 
into double points.
Kanenobu showed that $L[\varepsilon_2,...,\varepsilon_n]$ has $n+1$ components if and only if the associated 
chord diagram contains no intersection among the chords,  
see the proof of \cite[Lem. (3.7)]{Kanenobu}.  
If $T$ is non-planar, then it is not hard to see that, for any $(\varepsilon_j\in \{-1,1\},~j=1,2,...,n)$, 
the corresponding chord diagram contains such an intersection. Thus we have the conclusion.  
\end{proof}
\begin{remark}
Lemmas \ref{lem:homfly_planar} and \ref{lem:Pnonplanar} are related to the main results of \cite{Kanenobu} and \cite{horiuchi}.  
\end{remark}
\section{Milnor invariants} \label{sec:milnor}
\subsection{A short definition} 
Given an $n$-component link $L$ in $S^3$, denote by $\pi$ the fundamental group of $S^3\setminus L$, and by $\pi_q$ 
the $q$th subgroup of the lower central series of $\pi$.  
We have a presentation  of $\pi/ \pi_q$ with $n$ generators, given by a
choice of meridian $m_i$ of the $i$th component of $L$, $i=1,...,n$.  
So the longitude $\la_j$ of the $j$th component of $L$ ($1\le j\le n$) 
is expressed modulo $\pi_q$ as a word  
in the $m_i$'s (abusing notations, we still denote this word by $\la_j$).  
The \emph{Magnus expansion} $E(\la_j)$ of $\la_j$ is the formal power series in 
non-commuting variables $X_1,...,X_n$ obtained by 
substituting $1+X_i$ for $m_i$ and $1-X_i+X_i^2-X_i^3+...$ for $m_i^{-1}$, $1\le i\le n$.   

Let $I=i_1 i_2 ...i_{k-1} j$ be a sequence of elements of $\n$. 
Denote by $\mu_L(I)$ the coefficient of $X_{i_1}...X_{i_{k-1}}$ in the Magnus expansion $E(\la_j)$.  
\emph{Milnor invariant} $\ov{\mu}_L(I)$ is the residue class of $\mu_L(I)$ modulo the greatest common divisor of 
all $\mu_L(J)$ such that $J$ is obtained from $I$ by removing at least one index and permuting 
the remaining indices cyclicly.  
The indeterminacy comes from the choice of the meridians $m_i$.  Equivalently, it comes from the indeterminacy of 
representing the link as the closure of a string link \cite{HL}.  
Let us recall below the definition of these objects. 
\subsection{String links} \label{sec:sl}
Let $n\ge 1$, and let $D^2\subset R^2$ be the two-dimensional disk equipped with $n$ marked points $x_1,...,x_n$ 
in its interior, lying on the diameter on the $x$-axis of $R^2$.
An \textit{$n$-string link}, or $n$-component string link, is the image of a proper embedding 
$\bigsqcup_{i=1}^n [0,1]_i \rightarrow D^2\times [0,1]$
of the disjoint union $\bigsqcup_{i=1}^{n} [0,1]_i$ of $n$ copies of $[0,1]$ in $D^2\times [0,1]$, 
such that for each $i$, the image of $[0,1]_i$ runs from $(x_i,0)$ to $(x_i,1)$.   
Each string of an $n$-string link is equipped with an (upward) orientation. 
The $n$-string link $\{x_1,...,x_n\}\times[0,1]$ in $D^2\times[0,1]$ is called 
the {\em trivial $n$-string link} and is denoted by $\1_n$.

For each marked point $x_i\in D^2$, there is a point $y_i$ on $\partial D^2$ in the upper half of $R^2$  
such that the segment $p_i=x_iy_i$ is vertical to the $x$-axis, as illustrated in Figure \ref{fig:disk}. 
Given an $n$-string link $L=\bigcup_{i=1}^n L_i$ in $D^2\times [0,1]\subset R^2\times [0,1]$, 
the \emph{closure} $\hat{L}$ of $L$ is the $n$-component link defined by 
$\hat{L}=\bigcup_{i=1}^n \hat{L_i}=L\cup(\bigcup_{i=1}^n(p_i\times\{0,1\})\cup(y_i\times I))$.   

The set of isotopy classes of $n$-string links fixing the endpoints has a monoid structure, 
with composition given by the \emph{stacking product} and with the trivial $n$-string link $\1_n$ as unit element. 
Given two $n$-string links $L$ and $L'$, we denote their product by $L\cdot L'$, which is obtained by stacking 
$L'$ \emph{above} $L$ and reparametrizing the ambient cylinder $D^2\times I$.

Habegger and Lin showed that Milnor invariants are actually well defined integer-valued invariants of string links \cite{HL}.  
(We refer the reader to \cite{HL} or \cite{yasuharaAGT} for a precise definition of Milnor invariants $\mu(I)$ of string links.)  
Furthermore, Milnor invariants of length $k$ are known to be finite-type invariants of degree $k-1$ for string links \cite{BN2,Lin}.  
As a consequence, Milnor invariants of length $k$ for string links are invariants of $C_{k}$-equivalence.  
Habiro showed that the same actually holds for Milnor invariants of links \cite{H}.  
\subsection{Some results}\label{sec:results_milnor}
It was shown by Habegger and Lin that Milnor invariants without repeated indices classify string links up to link-homotopy \cite{HL}.   
Here, the link-homotopy is the equivalence relation generated by self-crossing changes. 
In \cite{yasuhara}, the second author gave an explicit representative for the link-homotopy class of any $n$-string link in terms of linear tree claspers. We shall make use of this representative in this paper, and recall its definition below. 

Let $\mathcal{J}_k$ denote the set of all sequences $j_0 j_1...j_k$ of $k+1$ non-repeating integers from $\{1,...,n\}$ such that $j_0<j_m<j_k$ for all $m$. 
Let $i_0 i_1...i_k$ be a sequence of $k+1$ integers from $\{1,...,n\}$ such that $i_0<i_1<\cdots<i_{k-1}<i_k$, 
and let $a_J$ be a permutation of $\{i_1,...,i_{k-1}\}$.  
Then $J=i_0 a_J(i_1)...a_J(i_{k-1})i_k$ is in $\mathcal{J}_k$ 
(and all elements of $\mathcal{J}_k$ can be realized in this way). 
Let $T_J$ be the simple linear $C_k$-tree for $\1_n$ as illustrated in Figure \ref{fig:TJ}. 
Here, $a_J$ is the unique positive $k$-braid which defines the permutation $a_J$ and such that every pair of strings crosses at most once.  
In the figure, we also implicitely assume that all edges of $T_J$ \emph{overpass} all components of $\1_n$. 
(This assumption is crucial in the computation of Milnor invariants.)  
 \begin{figure}[!h]
  \includegraphics{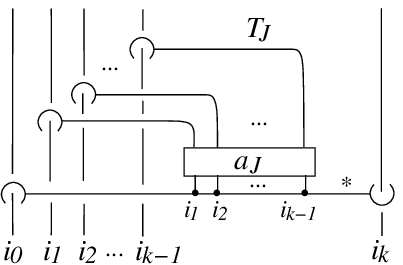}
  \caption{The $C_k$-trees $T_{J}$ and $\overline{T}_{J}$. }\label{fig:TJ}
 \end{figure}
Let $\overline{T}_J$ be the $C_k$-tree obtained from $T_J$ by inserting a positive half-twist in the $\ast$-marked edge, see Figure \ref{fig:TJ}. 
Denote respectively by $V_J$ and $V_J^{-1}$ the $n$-string links obtained from $\1_n$ by surgery along $T_{J}$ and $\overline{T}_J$.  
This notation is justified by the fact that, for any $J$ in $\mathcal{J}_k$, the string link  $V_J\cdot V_J^{-1}$
is $C_{k+1}$-equivalent to the trivial one \cite{H}.  
\begin{theorem}[\cite{yasuhara}] \label{thm:repr}
Any $n$-string link $L$ is link-homotopic to $l=l_1\cdots l_{n-1}$, where 
\[
  l_i= \prod_{J\in \mathcal{J}_i} V_J^{x_J}\textrm{ , where }\quad
x_{J}=\mu_{l_i}(J)=\left\{\begin{array}{ll}
\mu_L(J)&\textrm{if $i=1$},\\ 
& \\
\mu_L(J)-\mu_{l_1\cdots l_{i-1}}(J)& 
\textrm{if $i\geq 2$}.
\end{array}\right.
\]
\end{theorem}

\begin{remark}\label{rem:VJ}
The above statement slightly differs from the one in \cite{yasuhara}. There, another family of $C_k$-trees is used in place of $T_J$ and 
$\overline{T}_J$.  However, the present statement is shown by the exact same arguments as for \cite[Thm. 4.3]{yasuhara}, since for any 
$J,J'\in \mathcal{J}_k$ ($k\ge 1$), we have 
\begin{equation}\label{eq:muVJ}
\mu_{V_{J}}(J')=
\left\{\begin{array}{ll}
1&\text{ if $J=J'$,}\\
0&\text{ otherwise, }
\end{array}
\right.
\end{equation}
\noindent (compare with \cite[Lem. 4.1]{yasuhara}). 
\end{remark}
For the representative $l=l_1\cdots l_{n-1}$ above, we have the following lemmas. 

\begin{lemma}\label{lem:1}
Let $I<12...n$ with $|I|=m\leq n$. 
Then, 
\begin{equation}\label{eq:muIln}
 \mu_{l}(I) \equiv x_I \textrm{ $($mod $\gcd \{ x_{J}~|~J\lneq I \})$}. 
\end{equation} 
Moreover, for all $k<m-1$, we have 
\begin{equation}\label{eq:muIln2}
 \mu_{l_1\cdot...\cdot l_{k}}(I)\equiv  0 \textrm{ $($mod $\gcd \{ x_{J}~|~J\lneq I,~|J|\leq k+1 \})$}. 
\end{equation}
\end{lemma}
\begin{proof}
By \cite[Lem. 3.3]{MY1} and Theorem~\ref{thm:repr}, 
we have $\mu_{l}(I)=\mu_{l_1\cdot...\cdot l_{m-2}}(I) + \mu_{l_{m-1}}(I)=\mu_{l_1\cdot...\cdot l_{m-2}}(I) + x_I$.  
Hence Equation (\ref{eq:muIln2}) implies  Equation (\ref{eq:muIln}), and it suffices to prove 
Equation (\ref{eq:muIln2}).
 
Note that, for an $n$-string link $L=\bigcup_{i=1}^n L_i$, we have $\mu_I(L)=\mu_I(\bigcup_{i\in\{I\}}L_i)$. 
Hence we may assume that $I=12...m$ and that $l$ is an $m$-string link. 
The result is shown by an analisis of the the Magnus expansion of a longitude of each ``building block'' $V_J^{x_J}$, 
for all $k< m-1$ and all $J\in \mathcal{J}_k$.  
Since we are aiming at computing Milnor invariant $\mu(I)$, we compute up to terms $O(I)$ involving monomials $X_{i_1}X_{i_2}...X_{i_p}$ 
such that $i_1 i_2... i_p$ is not a subsequence of $12...(m-1)$. 
(Note that $O(I)$ includes any monomial where some variable appears at least twice, as well as any monomial involving $X_m$.)
For a subset $\{K\}$ of $\{I\}$, we will also use the notation $M(\{K\})$ for a sum of terms involving monomials such that all $X_j~(j\in\{K\})$ 
appear exactly once in each monomial. 

Let $J=j_0j_1...j_k\in \mathcal{J}_k$, for some $k< m-1$. 
Let $j$ be an index in ${I}$, and denote by $\la_j$ the $j$th longitude of $V_J^{\pm 1}$. 
Notice that all monomials appearing in $E(\la_{j})$ are in the variables $X_i$ such that $i\in \{J\}$, since 
all edges of both $T_J$ and $\overline{T}_J$ overpass all components of $\1_n$. 
By our assumption, $\{J\}$ is a subset of $\{I\}$, and in particular $j_k\leq m$.
Then there are three cases : 
\begin{enumerate}
\item[(i)] If $j<j_0$ or $j_k< j$, then clearly we have $E(\la_j)=1$.  
\item[(ii)] If $j\in \{J\}$, since all Milnor invariants of $V_J^{\pm 1}$ with length at most $k$ vanish,  
 $$ E(\la_{j}) = 1 +M(\{J\}\setminus \{j\}) + M(\{J\})+O(I). $$
\item[(iii)] If $j\notin \{J\}$ and $j_0< j <j_{k}$, then 
since $V_J^{\pm 1}\setminus (i\text{th component})$ is trivial for any $i\in\{J\}$, 
all (nontrivial) monomials appearing in $E(\la_{j})$ contain all variables $X_i~(i\in \{J\})$.  
Hence we have a Magnus expansion of the form
$$
 E(\la_{j}) = 1 + M(\{J\}) + O(I).
$$
\end{enumerate}
Summarizing all three cases, we have  
\begin{equation}\label{eq:Mj}
 E(\la_{j}) =
 \left\{\begin{array}{ll}
 1 & \text{if $i<j_0$ or $j_k< j$,}\\
 1 + M(\{J\}\setminus\{j\}) + M(\{J\})+O(I)  & \text{if $j\in \{J\}$, }\\
 1 + M(\{J\}) + O(I)  & \text{if $j\notin \{J\}$ and $j_0< j <j_{k}$.}
\end{array}
\right.
\end{equation} 
\noindent 
In particular, for $j=m$ we are either in case (i) or (ii), depending on whether $m$ is in $\{J\}$ or not.  
Note that case (ii) occurs only when $m=j_k$. 
Since $j_0$ is the smallest integer in $\{J\}$, any monomial in $M(\{J\}\setminus\{m\})$ 
whose leftmost variable is not $X_{j_0}$ is in $O(I)$. 
So it follows from (\ref{eq:muVJ}) that if $m$ is in $\{J\}$ we have 
 $$ E(\la_{m}) = 1 \pm X_{j_{0}}...X_{j_{k-1}}+ O(I).$$
\noindent (Note that $M(\{J\})$ is contained in $O(I)$ if $m\in \{J\}$.) 
 
Now let us consider the stacking product $V_J^{\pm 1}\cdot V_J^{\pm 1}$. 
The Magnus expansion of the $j$th longitude of $V_J^{\pm 1}\cdot V_J^{\pm 1}$ is given by 
$E(\la_j)E(\tilde{\la_j})$, where $\la_j$ is the $j$th longitude of $V_J^{\pm 1}$ and 
$E(\tilde{\la_j})$ is obtained from $E(\la_j)$ by replacing $X_i$ with 
$E(\la_i)^{-1}X_iE(\la_i)$ for each $i\in\{I\}$. 
By (\ref{eq:Mj}), we have
$$E(\la_i)^{-1}X_iE(\la_i)=
 \left\{\begin{array}{ll}
 X_i & \text{if $i<j_0$ or $j_k< i$,}\\
 X_i+ M(\{J\})+O(I)  & \text{if $i\in \{J\}$, }\\
 X_i+M(\{J\}\cup\{i\})+ O(I)  & \text{if $i\notin \{J\}$ and $j_0< i <j_{k}$.}
\end{array}
\right.
$$
This implies that
$$ E(\la_j)E(\tilde{\la_j})=
 \left\{\begin{array}{ll}
1 & \text{if $j<j_0$ or $j_k<j$,}\\
E(\la_j)E(\la_j)+ M(\{J\})+O(I)&\text{if $j\in \{J\}$}.\\
1+M(\{J\})+O(I)& \text{if $j\notin \{J\}$ and $j_0< j <j_{k}$}.
\end{array}
\right.
$$
It follows that the Magnus expansion of the $j$th longitude $\la_{J,j}$ in $V_J^{x_J}$ is given by 
$$ E(\la_{J,j})=
 \left\{\begin{array}{ll}
1 & \text{if $j<j_0$ or $j_k<j$,}\\
E(\la_j)^{|x_J|}+ M(\{J\})+O(I)&\text{if $j\in \{J\}$}.\\
1+M(\{J\})+O(I)& \text{if $i\notin \{J\}$ and $j_0< i <j_{k}$}.
\end{array}
\right.
$$

In particular, the Magnus expansion of the $m$th longitude $\la_{J,m}$ in $V_J^{x_J}$ is either $1$ 
if $m\notin\{J\}$, or 
$1 +x_JX_{j_{0}}...X_{j_{k-1}}+ O(I)$ otherwise (in which case $j_k=m$). 
(Recall that if $m\in \{J \}$, each term in $M(\{J\})$ involves the variable $X_m$, and hence is in $O(I)$.)  
It follows that we have 
$$
E(\la_{J,m})=
\left\{\begin{array}{ll}
1 & \text{if $m\notin\{J\}$}\\
1 +O(I)~(\text{mod}~x_J) & \text{if $m\in \{J\}$ and $J<I$,}\\
1 +O(I) & \text{otherwise}.
\end{array}
\right.
$$
We can now focus on the computation of $\mu_{l_1\cdot...\cdot l_k}(I)$.  
Recall that the Magnus expansion of the $m$th longitude of $l_1\cdot...\cdot l_{k}$ is obtained from 
a product of $E(\la_{J,m})$'s ($\{J\}\subset \{I\},~J\in\bigcup_{s=1}^k\mathcal{J}_s$) 
by replacing each variable $X_i$ with $X_i+(\text{monomials involving }X_i)$.
It follows that the Magnus expansion of the $m$th longitude of $l_1\cdot...\cdot l_{k}$ 
is of the form 
\[1+O(I) ~~~~~~\text{ $($mod $\gcd \{ x_{J}~|~J\lneq I,~|J|\leq k+1 \})$.}\] 
This implies Equation (\ref{eq:muIln2}). 
\end{proof}

\begin{lemma}\label{lem:2}
Let $I<12...n$ with $|I|=m\leq n$. Then 
 $$ \Delta_{l}(I) = \gcd \{ x_J~|~J\lneq I\}. $$
\end{lemma}
\begin{proof}
The proof is by induction on $m$.  
For $m=3$, the result is clear since $x_{ij}=\mu_{l_1}(ij)=\mu_l(ij)$ for any $i,j$.    
Now, let $m\ge 4$.  
It will be convenient to use the notation 
$\delta_{k}(I)$ for the set of all sequences of length $(m-k)$ obtained from $I$ by removing $k$ indices and permuting cyclicly. 
By definition, 
$$ \Delta_{l}(I) = \gcd ( \{ \mu_l(J)~|~J\in \delta_k(I),~k>1 \} \cup \{ \mu_l(J)~|~J\in \delta_{1}(I) \} ). $$
By the induction hypothesis, we have that 
\beQ
\gcd \{ \mu_l(J)~|~J\in \delta_k(I),~k>1 \} & = & 
       \gcd \{ \Delta_l(J)~|~J\in \delta_{1}(I) \}    \\
 & = & \gcd \{ \Delta_l(J)~|~J\in \delta_{1}(I),~J<I \} \\
 & = & \gcd \{ x_{J'}~|~J'\lneq J,~J\in \delta_{1}(I),~J<I \} \\
 & = & \gcd \{ x_{J'}~|~J'<I,~|J'|<m-1 \}.  
\eeQ
On the other hand, by Lemma \ref{lem:1}, for all $J\in \delta_{1}(I)~(J<I)$ and for any sequence $\tau(J)$ 
obtained from $J$ by permuting cyclicly, we have 
\[ \mu_{l}(\tau(J)) \equiv \mu_{l}(J) \equiv x_J \textrm{ (mod  $\gcd \{ x_{J'}~|~ J'\lneq J \})(=\Delta_l(J))$}. \]
It follows that $\Delta_{l}(I) = \gcd \{ x_J ~|~J<I,~|J|\le m-1 \}$,
as desired.  
\end{proof}
\section{Proof of Theorem \ref{thm:main}}\label{sec:proof}
Let $L=\bigcup_{i=1}^n L_i$ be an  $n$-component link in $S^3$, and let 
$I$ be a sequence of $(m+1)$ distinct elements of $\{1,...,n\}$. 
It is sufficient to consider here the case $m+1=n$, since, if $m+1<n$, we have that $\ov{\mu}_L(I)=\ov{\mu}_{\bigcup_{i\in\{I\}}L_i}(I)$. 
We may further assume that $I=12...n$ without loss of generality. 
Indeed, for any permutation $I'$ of $12...n$, we have that $\ov{\mu}_L(I')=\ov{\mu}_{L'}(12...n)$, 
where $L'$ is obtained from $L$ by reordering the components appropriately.  

We first show how to reformulate the problem in terms of string links.  

\subsection{Closing string links into knots}\label{sec:closesl}
Let $B_I$ be an $I$-fusion disk for $L$, as defined in the introduction.  
Up to isotopy, we may assume that the $2n$-gon $B_I$ lies in the unit disk $D^2$ as shown in Figure \ref{fig:disk}, 
where the edges $p_j$ are defined by $p_j=x_jy_j$, $1\le j\le n$.  
We may furthermore assume that $L\cup B_I$ lies in the cylinder $D^2\times [0,1]$,   
such that $B_I\subset (D^2\times \{ 0 \})$, and such that
$$L\cap \partial (D^2\times [0,1]) 
= \bigcup_{j=1}^n \left ( (p_j\times \{0\})\cup (\{ y_j \}\times [0,1])\cup (p_j\times \{1\}) \right). $$ 
 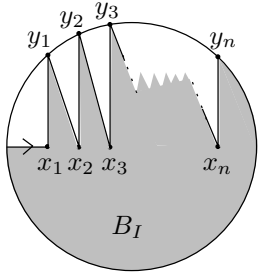
\begin{figure}[!h]
  \input{disk.pstex_t}
  \caption{The $2n$-gon $B_I$ lying in the unit disk $D^2$. }\label{fig:disk}
 \end{figure}

\noindent 
In this way, we obtain an $n$-string link $\si$ whose closure $\hat{\si}$ is the link $L$, by setting 
\begin{equation}\label{eq:defsigma}
 \si := \overline{L\setminus (L\cap \partial (D^2\times [0,1]))}. 
\end{equation}

Given an $n$-string link $K=\bigcup_{i=1}^n K_i$ and any subsequence $J$ of $I=12...n$, 
we will denote by $K(J)$ the knot 
$$ K(J):=\overline{( (\bigcup_{j\in \{J\}} \hat{K}_j)\cup \partial B_I )\setminus ( (\bigcup_{j\in \{J\}} \hat{K}_j)\cap B_I )}. $$ 
Note that $K(J)$ coincides with the knot $(\hat{K})_J$ defined in the introduction for the choice of $I$-fusion disk $B_I$ specified above.  

Recall from Section \ref{sec:results_milnor} that, for any $k$ and any $J\in \mathcal{J}_k$, 
$V_J$ (resp. $V_J^{-1}$) denotes the $n$-string link obtained from $\1_n$ by surgery along the $C_k$-tree $T_J$ 
(resp. $\ov{T}_J$), see Figure \ref{fig:TJ}.  
Denote by $t_J$ (resp. $\ov{t}_J$) the image of the $C_k$-tree $T_J$ (resp. $\ov{T}_J$) for $\1_n$ under taking 
the closure $V_J^{\pm 1}(I)$.  
We observe that $t_J$ (resp. $\ov{t}_J$) is a planar tree clasper for the unknot if and only if $J<12...n$.  
In this case, note that $V_J(I)$ (resp. $V_J^{- 1}(I)$) is the knot $K^\varepsilon_{k-1}$ of Lemma \ref{fig:Kn} 
with $\varepsilon=(-,+,...,+)$ (resp. for $\varepsilon=(+,+,...,+)$). 
In particular, observe that $V_I^{x_I}(J)$ is the unknot for all $J\lneq I$ 
and that, by Lemma \ref{lem:parallel} (for $m=1$), the knot 
$V_I^{x_I}(I)$ is $C_{n}$-equivalent to the connected sum of 
$|x_I|$ copies of $V_I^{\ve_I}(I)$, where $\ve_I$ denotes the sign of $x_I$. 
By Lemma \ref{lem:homfly_planar} we thus have, for all $I,J\in\mathcal{J}_{n-1}$, 
\begin{equation}\label{eq:pVI}
(\log P_0(V^{x_J}_J(I)))^{(n-1)}=  P_0^{(n-1)}(V^{x_J}_J(I)) = \left\{\begin{array}{ll}
   (-1)^{n-1}x_I(n-1)!2^{n-1} & \textrm{if $J=I$},\\ 
   0 & \textrm{otherwise.}
\end{array}\right.
\end{equation}

\subsection{Proof of Theorem \ref{thm:main}}
Let $\si$ be the $n$-string link with closure $L$ defined in Section \ref{sec:closesl}.  
By Theorem~\ref{thm:repr}, 
$\si$ is link-homotopic to $l_k\cdot ...\cdot l_{n-1}$, 
where $l_i=\prod_{J\in \mathcal{J}_i} V_J^{x_J}$ is defined in Section \ref{sec:results_milnor}, 
and with $n\leq 2k+1$ (by our vanishing assumption on Milnor invariants).
Hence $\si$ is obtained from $l_k\cdot ...\cdot l_{n-1}$ by surgery along a disjoint union $R_1$ 
of simple $C_1$-trees whose leaves intersect a single component of $l_k\cdot ...\cdot l_{n-1}$.  

By Lemma \ref{lem:parallel}, for all $J<I$, we have that 
 $$ \si(J)\stackrel{C_{n}}{\sim} l_{n-1}(J)\sharp (l_k\cdot ...\cdot l_{n-2})_{R_1}(J). $$
Since $(\log P_0)^{(n-1)}$ is an invariant of $C_{n}$-equivalence for all $n$, 
it follows from the additivity property of $(\log P_0)$ that 
 $$ \big(\log P_0(\si(J))\big)^{(n-1)} 
  = \big(\log P_0(l_{n-1}(J))\big)^{(n-1)} + \big(\log P_0\left( (l_k\cdot ...\cdot l_{n-2})_{R_1}(J)\right)\big)^{(n-1)}.   
 $$
The proof of the next lemma is postponed to Section \ref{sec:proofclaim}. 

\begin{lemma}\label{lem:claim}
 \[ \frac{-1}{(n-1)!2^{n-1}}\sum_{J<I} 
(-1)^{|J|}
  \big(\log P_0\left( (l_k\cdot ...\cdot l_{n-2})_{R_1}(J)\right)\big)^{(n-1)} 
 \equiv 0  \textrm{ (mod $\Delta_{L}(I)$)}. \]
\end{lemma}

\medskip
\noindent 
It follows that 
\beQ
&&
\frac{-1}{(n-1)!2^{n-1}}\sum_{J<I} (-1)^{|J|}
 \big(\log P_0(\si(J))\big)^{(n-1)} \\
 & \equiv & 
\frac{-1}{(n-1)!2^{n-1}}\sum_{J<I} (-1)^{|J|}
 \big(\log P_0(l_{n-1}(J))\big)^{(n-1)} \textrm{ (mod $\Delta_{L}(I)$)} \\
 & \equiv & 
\frac{(-1)^{n-1}}{(n-1)!2^{n-1}}\big(\log P_0(V_I^{x_I}(I))\big)^{(n-1)} \textrm{ (mod $\Delta_{L}(I)$)} \\
 & \equiv & 
x_I \textrm{ (mod $\Delta_{L}(I)$)},
\eeQ
where the second equality holds by Lemma \ref{lem:Pnonplanar}, and the third one follows from Equation (\ref{eq:pVI}). 

On the other hand, by Lemmas \ref{lem:1} and \ref{lem:2}, we have
 $$ \mu_{L}(I) =\mu_\si(I)= \mu_{l_k\cdot ...\cdot l_{n-1}}(I) 
   \equiv x_I \textrm{ (mod $\Delta_{L}(I)$)}, $$
which completes the proof.  

\subsection{Proof of Lemma \ref{lem:claim}}\label{sec:proofclaim}

First, it is convenient to slightly modify the string link $(l_k\cdot ...\cdot l_{n-2})_{R_1}$.  
For that purpose, we regard it as obtained from $\1_n$ by surgery along the disjoint union of tree claspers $G\cup R_1$, with 
$$ G := \bigcup_{i=k}^{n-2} \left(\bigcup_{J\in \mathcal{J}_i} T^{x_J}_J \right), $$   
where $T_J^{x_J}$ denotes $|x_J|$ parallel copies of $T_J$ (resp. $\overline{T}_J$) if $x_J>0$ (resp. if $x_J<0$).

A tree clasper for $\1_n$ is said to be \emph{in good position} if, in the usual diagram of $\1_n$, each component of 
$\1_n$ underpasses all edges of the tree clasper. 
For example, each component of $G$ is in good position (see Figure \ref{fig:TJ}), whereas the components of $R_1$ may not be. 
However, by repeated applications of \cite[Prop. 4.5]{H} we have 
 $$ (\1_n)_{G\cup R_1}\stackrel{C_n}{\sim} (\1_n)_{G\cup \tilde{R}}, $$
where $\tilde{R}$ is a disjoint union, disjoint from $G$, of 
simple tree claspers for $\1_n$ in good position and intersecting some component of $l_k\cdot ...\cdot l_{n-1}$ more than once.

We now close the string link $(\1_n)_{G\cup \tilde{R}}$ using the sequence $I=12...n$, as explained in Section \ref{sec:closesl}.  
It follows from Lemma \ref{lem:Pnonplanar} that, for all $J<I$, we have 
 $$ \big( \log P_0((\1_n)_{G\cup \tilde{R}}(J)\big)^{(n-1)} = \big( \log P_0( (\1_n)_{\tilde{G}\cup \tilde{R}}(J)\big)^{(n-1)}, $$
where 
 $$ \tilde{G} := \bigcup_{i=k}^{n-2} \left( \bigcup_{J\in \mathcal{J}_i\textrm{ ; }J<I} T^{x_J}_J \right). $$
In other words, we only need to consider those tree claspers $T_J$ and $\ov{T}_J$ with $J<I$, since only those become planar under closure. 
Moreover, since $\Delta_L(I)$ divides all $x_J$ with $J<I$, we can express each $T^{x_J}_J$ as a disjoint union of 
parallel trees with multiplicity $\Delta_L(I)$.  
The knot $(\1_n)_{\tilde{G}\cup \tilde{R}}(I)$ is obtained from $U$ by surgery along a disjoint union of tree claspers 
 $$ F:=t\cup r\quad \textrm{ ; }\quad t:=\bigcup_{i=1}^q t_i\quad,\quad r:=\bigcup_{j=1}^p r_j, $$
for some integers $q,p$, where $r$ is the image of $\tilde{R}$ under closure and where, for each $i=1,...,p$, 
the clasper $t_i$ is a parallel familly of $\Delta_L(I)$ copies of (the image under closure of) 
some $C_m$-tree $T_J$ or $\ov{T}_J$ with $J<I$ ($k\le m\le n-1$).

We need the following additional definition. 
A disjoint union $C_1\cup ...\cup C_s$ of $s\ge 1$ (possibly parallel) tree claspers 
of degree $<n$ 
for $U$ is called \emph{balanced} if each tree $C_j$ 
is being assigned a subset $w(C_j)$ of $\n$, called \emph{weight}, such that 
\begin{equation}\label{eq:balanced}
 (\1_n)_{\tilde{G}\cup \tilde{R}}(J)\stackrel{C_n}{\sim} U_{\big( \bigcup_{w(C_i)\subset \{J\}} C_i \big)},   
\end{equation}
for all $J<I$. 
For example, $F$ is balanced if we assign the index of each tree as weight.  
We say that a $C_k$-tree in a balanced family is \emph{repeated} if its weight has at most $k$ elements, 
that is, if its weight is smaller than the number of leaves.  
For instance, all tree claspers $r_j~(j=1,...,p)$ are repeated.

Now, \emph{up to $C_n$-equivalence}, we deform $U_F$ into a connected sum of 
knots obtained from $U$ by surgery along 
a single (possibly parallel) tree clasper. 
In other words, we will deform $F$ into a balanced union of \emph{localized} tree claspers for $U$, 
i.e. tree claspers sitting in a $3$-ball that intersects $U$ at a single strand and is disjoint from all other tree claspers.  
Since we started with tree claspers in good position for $\1_n$, 
this deformation can be achieved, starting from $F$, by a sequence of only leaf slides and edge crossing changes, see Figure \ref{fig:slide}. 
By Lemma \ref{lem:parallel}, performing such operations may introduce additionnal tree claspers up to $C_n$-equivalence.  
However, the following is easily verified. 

\begin{claim}\label{fact}
Let $M\cup T\cup T'$ be a balanced union of tree claspers for $U$, where $T\cup T'$ is as in Lemma \ref{lem:parallel}. 
If, in the statement of Lemma \ref{lem:parallel}(1) (resp. of Lemma \ref{lem:parallel}(2)),  
we assign the weights $w(T)$ and $w(T')$ to $\tilde{T}$ and $\tilde{T'}$ respectively, 
and the weight $w(T)\cup w(T')$ to $Y$ (resp. $H$) and each connected component of $C$, 
then $M\cup\tilde{T}\cup \tilde{T'}\cup Y\cup C$ (resp. $M\cup\tilde{T}\cup \tilde{T'}\cup H\cup C$) is balanced.  
In particular, if the degree of $T$ is at least $(n-1)/2$, where $n$ is the number of strands of $L$,  
then all tree claspers in $C$ are repeated. 
\end{claim}

We now start our localization process, which goes in three steps.

The first step consists in localizing all parallel trees $t_i$.  
Consider, say, the parallel $C_{k_1}$-tree $t_1$.  
Then by assumption we have that $k_1\ge (n-1)/2$, 
and Lemma~\ref{lem:parallel} and Claim~\ref{fact} imply that 
$U_F\stackrel{C_n}{\sim} U_{t_1}\sharp U_{(F\setminus t_1)\cup F_1}$, 
where $F_1$ is a disjoint union of tree claspers of degree $> k_1$, which are either repeated trees 
or parallel trees with multiplicity $\Delta(I)$.
Using this argument repeatedly, we see that $F$ can be be deformed into a balanced union of tree claspers 
$F'=h_1\cup...\cup h_l\cup r'$, for some integer $l$, 
where $r'$ is a disjoint union of repeated trees and each $h_i$ is a parallel tree with multiplicity $\Delta_{L}(I)$,  
such that 
$$ (\1_n)_{\tilde{G}\cup \tilde{R}}(I)=U_F\stackrel{C_n}{\sim} U_{F'}=U_{h_1}\sharp ...\sharp U_{h_l} \sharp U_{r'}. $$  

In the second step, we ``split'' each parallel tree into $\Delta_L(I)$ localized ones.  
Indeed, since each $h_i$ is a parallel family of $\Delta_{L}(I)$ copies of some tree clasper $h'_i$, 
we can apply Lemma~\ref{lem:parallel} (with $m=1$) and Claim~\ref{fact}  
to deform $F'$ into a balanced family
  $$ F''=\bigcup_{i=1}^l (\underbrace{h'_i\cup...\cup h'_i}_{\textrm{$\Delta_{L}(I)$ times}})\cup r'', $$
where for each $i$ the tree clasper $h'_i$ has weight 
$w(h'_i)=w(h_i)$ and where $r''$ is a disjoint union of repeated trees, such that 
$$ U_{F'}\stackrel{C_n}{\sim} U_{F''}=
 \big(\Delta_{L}(I)\times U_{h'_1}\big)\sharp ...
 \sharp \big( \Delta_{L}(I)\times U_{h'_l}\big) \sharp U_{r''}.  
$$
(Here $\Delta_{L}(I)\times U_{h'_i}$ denotes the connected sum of $\Delta_{L}(I)$ copies of $U_{h'_i}$ ; $(i=1,...,l)$.

In the third and last step, we localize all repetead trees in $r''$.  
Note that, by Claim~\ref{fact}, performing a leaf slide or an edge crossing change between two repeated tree claspers 
only introduces new tree claspers that are also repeated. 
Hence $F''$ can be deformed into a balanced disjoint union of tree claspers
 $$ X = \bigcup_{i=1}^l (\underbrace{h'_i\cup...\cup h'_i}_{\textrm{$\Delta_{L}(I)$ times}})\cup \bigcup_{j=1}^{l'}  x_j, $$ 
for some integer $l'$, where each $x_j$ is a repeated tree clasper, such that 
\begin{equation}\label{eq:local}
 U_{F''}\stackrel{C_n}{\sim} U_{X}=
 \big(\Delta_{L}(I)\times U_{h'_1}\big)\sharp\cdots
 \sharp \big( \Delta_{L}(I)\times U_{h'_l}\big) \sharp U_{x_1}\sharp\cdots\sharp U_{x_{l'}}.  
\end{equation} 
This concludes the localization process.

Now, since $X$ is balanced, and since $\log P_0$ is additive under connected sum (see Section \ref{sec:homfly}), 
for any $J<I$ we have 
 \[\begin{array}{l}
(\log P_0((\1_n)_{\tilde{G}\cup \tilde{R}}(J)))^{(n-1)}\\
\displaystyle \hspace*{2cm}=\Delta_{L}(I)\sum_{w(h'_i)\subset \{J\}}(\log P_0( U_{h'_i}))^{(n-1)}
+\sum_{w(x_j)\subset \{J\}}(\log P_0( U_{x_j}))^{(n-1)},
\end{array}\]
where the first (resp. second) sum is over all tree claspers $h'_i$ (resp. $x_j$) 
whose weight is contained in $\{ J \}$.  
On the other hand, we have the following
 
\begin{claim}\label{claim:R}
Let $g$ be a connected component of $X$. \\
(1)~If $|w(g)|<n(=|I|)$, then 
$$ \sum_{J<I\textrm{ ; } w(g)\subset\{J\}} 
(-1)^{|J|}
(\log P_0 (U_g))^{(n-1)}=0. $$
(2)~If $|w(g)|=n$ (i.e. $g=h'_i$ for some $i$ and $g$ is a $C_{n-1}$-tree), then
 $$ (\log P_0 (U_g))^{(n-1)}\equiv 0~~ \text{mod}~~(n-1)!2^{n-1}. $$
 \end{claim}
Note that, since any connected component $g$ of $X$ has degree $<n$, we have that $|w(g)|<n$ if $g$ is repeated.  
Hence it follows from Claim \ref{claim:R} that 
 \[ 
\frac{-1}{(n-1)!2^{n-1}}\sum_{J<I} (-1)^{|J|}
 (\log P_0((\1_n)_{\tilde{G}\cup \tilde{R}}(J)))^{(n-1)} \equiv 0 \textrm{ (mod $\Delta_{L}(I)$)}, \]
which concludes the proof of Lemma~\ref{lem:claim}.  

\begin{proof}[Proof of Claim \ref{claim:R}]
(1)~Since $w(g)<n$, there is an element $a\in\{I\}$ such that 
$a\notin w(g)$. We denote by $I\setminus a$ the sequence obtained from $I$ by deleting $a$. Then we have that 
$$ 
\sum_{J<I\textrm{ ; } w(g)\subset\{J\}}  
(-1)^{|J|}
=
\sum_{J<I\setminus a\textrm{ ; } w(g)\subset\{J\}}
(-1)^{|J|}
 +
 \sum_{J<I\setminus a\textrm{ ; } w(g)\subset\{J\}\cup\{a\}}
(-1)^{|J|+1}
 = 0,
$$
which implies the desired equality.

(2)~Using the AS and IHX relations for tree claspers (see \cite{G,H}), one can check that the knot $U_g$ is $C_n$-equivalent 
to a connected sum of knots $U_{g_i}$, where each $g_i$ is a linear
$C_{n-1}$-tree which is either non-planar or of the form shown in Figure \ref{fig:knclasper}. 
Since $(\log P_0)^{(n-1)}$ is an invariant of $C_n$-equivalence, 
the result then follows from Lemmas \ref{lem:homfly_planar} and \ref{lem:Pnonplanar}. 
\end{proof}

\section{First non-vanishing Milnor invariants and link-homotopy of string links}\label{sec:log}

We begin this section by proving Theorem \ref{thm:main2}.  
Most of the arguments follow very closely the proof of Theorem \ref{thm:main}, 
and we therefore freely use the notions and results of the previous section.  

\subsection{Proof of Theorem \ref{thm:main2}}

Let $L=\bigcup_{i=1}^n L_i$ be an $n$-component link in $S^3$ with vanishing 
Milnor link-homotopy invariants of length up to $k$ $(3\leq k+1\leq n)$.  
Let $I$ be a sequence of $(k+1)$ distinct elements of $\{1,...,n\}$. 
As in Section \ref{sec:proof}, we may assume without loss of generality that $k+1=n$ and that $I=12...n$.  
Following Section \ref{sec:closesl}, we may also assume that the $2n$-gon $B_I$ is chosen so that $L\cup B_I$ lies in the cylinder 
$D^2\times [0,1]$,   
such that $B_I\subset (D^2\times \{ 0 \})$ is as shown in Figure \ref{fig:disk}.   
Hence (\ref{eq:defsigma}) defines an $n$-string link $\si$ whose closure is $L$. 

By Theorem~\ref{thm:repr}, the $n$-string link $\si$ is link-homotopic to 
$l_{n-1}=\prod_{J\in \mathcal{J}_{n-1}} V_J^{x_J}$ defined in Section \ref{sec:results_milnor}. 
By applying the exact same arguments as in Section \ref{sec:proofclaim}, 
there exists a disjoint union of tree claspers $R=r_1\cup ...\cup r_p$, 
with each being assigned a weight $w(r_i)\subset \{ 1,...,n \}$, 
such that 
\begin{itemize}
 \item for each $i$, we have $|w(r_i)|\le \deg(r_i)$, 
 \item $L_J\stackrel{C_{n}}{\sim} l_{n-1}(J)\sharp U_{R_J}$ for all $J<I$,
where $R_J=(\bigcup_{w(r_i)\subset J}{r_i})$. 
       (In particular, $R_I=R$ and $L_I$ is $C_{n}$-equivalent to  $l_{n-1}(I) \sharp U_{R}$.)
\end{itemize}
Since 
$l_{n-1}$ is $C_{n-1}$-equivalent to $\1_n$, 
we have by Equation (\ref{eq:nadditivity}) that, for all $J<I$,  
$$P_0^{(n-1)}(L_J)=P_0^{(n-1)}(l_{n-1}(J))+P_0^{(n-1)}(U_{R_J}).$$

The following claim is proved below.
\begin{claim}\label{claim:stringR}
$$\sum_{J<I}
(-1)^{|J|}
P_0^{(n-1)}(U_{R_J})=0.$$
\end{claim}

This and Equation (\ref{eq:pVI}) imply that 
\beQ
\displaystyle\frac{-1}{(n-1)!2^{n-1}}\sum_{J<I}
(-1)^{|J|} P_0^{(n-1)}(L_J) & 
= & \displaystyle\frac{-1}{(n-1)!2^{n-1}}\sum _{J<I}
(-1)^{|J|}P_0^{(n-1)}(l_{n-1}(J)) \\
 & = & \displaystyle\frac{-1}{(n-1)!2^{n-1}}\sum _{J<I}
(-1)^{|J|} P_0^{(n-1)}(V_I^{x_I}(J)) \\
 & = & \displaystyle\frac{(-1)^{n-1}}{(n-1)!2^{n-1}}P_0^{(n-1)}(V_I^{x_I}(I))=x_I 
\eeQ
where the second equality follows from  Lemma \ref{lem:Pnonplanar}. 
Lemma~\ref{lem:1} completes the proof.

\begin{proof}[Proof of Claim \ref{claim:stringR}]

We will show that the alternate sum 
$$\sum_{J<I} (-1)^{|J|} U_{R_J}$$
is a linear combination of singular knots with $n$ double points.  
Since $P^{(n-1)}_0$ is a finite type invariant of degree $n-1$, this implies Claim~\ref{claim:stringR}.  

We may assume without loss of generality that $\bigcup_{i} w(r_i)=\n$. 
Indeed, if there exists some $j\in \n$ such that $j\notin w(r_i)$ for all $i$, 
we can freely add a $C_1$-tree $c_j$ with weight $\{j\}$ such that 
$U_{R\cup c_j}=U_{R}\sharp U_{c_j}=U_{R}\sharp U$.   

For each $i=1,...,p$, let deg$(r_i)=d_i$.
Consider the $(d_i+1)$-component trivial tangle which is the intersection of $U$ with a regular neighborhood of $r_i$. 
Then surgery along $r_i$ yields a $(d_i+1)$-component tangle
 $ \beta^i=\beta^i_0\cup...\cup \beta^i_{d_i}$. 
Note that $\beta^i$ is a Brunnian tangle \cite{H}.  Since $\beta^i\setminus \beta^i_0$ is trivial, 
there is a diagram of $\beta^i$ such that, for all $u=1,...,d_i$, 
the component $\beta^i_u$ is a trivial arc that only crosses component $\beta^i_0$. 
Fix a diagram of $U_{R}$ that satisfies this condition for all $i=1,...,p$.  
Now, let $w(r_i)=\{j_1,...,j_{m_i}\}\subset \n$, with $m_i\leq d_i$, and 
for all $u\in \{1,..,m_i\}$.
Set 
\[
 \mathcal{S}_i(j_u):=
  \textrm{the set of all crossings where $\beta^i_0$ underpasses $\beta^i_u$ }.
\]
Note that this is only possible because $r_i$ satisfies $m_i=|w(r_i)|\le d_i$.  
For all $j\in \{1,..,n\}$, set 
$$ \mathcal{S}(j) = \bigcup_i \mathcal{S}_i(j). $$

For any $J<I$, denote by $U_{R}[J]$ the knot obtained from $U_{R}$ by switching all crossings in 
$\bigcup_{j\in \{J\}} \mathcal{S}(j)$.  
Then $U_{R}[J]$ is obtained from $U$ by surgery along all $r_i$ such that $w(r_i)\cap \{J\}=\emptyset$, i.e.,    
 $$ U_{R}[J]=U_{R_{I\setminus J}} $$
 for any $J<I$, where $I\setminus J$ denotes the sequence obtained from $I$ by deleting all $j\in \{J\}$. 
(In particular, we have $U_{R}[\emptyset]=U_{R}$.) 
Hence we have 
 $$ \sum_{J<I} (-1)^{|J|} U_{R_J} = \sum_{J<I} (-1)^{|J|} U_{R}[I\setminus J]=\sum_{J<I} (-1)^{n-|J|} U_{R}[J], $$
Clearly the alternate sum on the right-hand side, which involves knots that differ from one another by crossing changes on $n$ sets of crossings, 
can be written as a linear combination of singular knots with $n$ double points. 
This completes the proof of Claim \ref{claim:stringR}.  
\end{proof}

\subsection{Link-homotopy of string links}\label{sec:link-homotopy}

In this section, we give several interesting consequences of Theorem \ref{thm:main2} for Milnor invariants of string links.  

We first define an analogue for string links of the band sum operations on links given in the introduction.
Let $L$ be an $n$-string link. 
Recall from Sections \ref{sec:sl} and \ref{sec:closesl}
that, for each $i=1,...,n$, we pick a point $y_i\in \partial D^2$ 
and thus have a segment $p_i=x_iy_i\subset D^2$ (see Figure \ref{fig:disk}). 
Recall also that the closure of $L$ is defined by 
$\hat{L}=\bigcup_{i=1}^n \hat{L_i}=L\cup(\bigcup_{i=1}^n(p_i\times\{0,1\})\cup(y_i\times I))$.   

Let $I=i_1i_2...i_{m+1}$ be a sequence of $m+1$ distinct integers in $\{1,...,n\}$. 
We choose a $2(m+1)$-gon $B_I$ in $R^2\times(-\infty,0]$ such that 
$B_I\cap (R^2\times\{0\})=\bigcup_{i\in I} (p_{i}\times\{0\})$ is a set of 
$m+1$ non-adjacent edges and $p_{i_1}\times\{0\},...,p_{i_{m+1}}\times\{0\}$ 
appear in this order along the oriented boundary of $B_I$. 
As in the introduction, for any subsequence $J$ of $I$, 
we can define an oriented knot $\hat{L}_J$ as the 
closure of $((\bigcup_{j\in \{J\}} \hat{L}_j)\cup\partial B_I)\setminus ((\bigcup_{j\in \{J\}} \hat{L}_j)\cap B_I)$. 

Set 
\[f_{B_I}(L)=\frac{-1}{m!2^{m}}\sum_{J<I}(-1)^{|J|} P_0^{(m)}(\hat{L}_J).\]
This function depends on the choice of $B_I$. 
Hence, for every nonrepeated sequence $I$, we choose $B_I$ and fix it, to obtain an invariant of string links $f_{B_I}$. 

The following is a string link version of Theorem \ref{thm:main2}. 
\begin{theorem}\label{thm:stringlink}
Let $L$ be an $n$-string link with vanishing Milnor link-homotopy invariants of length $\le k$ ($3\le k+1\le n$).  
Then for any sequence $I$ of length $k+1$ without repeated indices, we have 
$\mu_L(I)=f_{B_I}(L)$. 
\end{theorem}
\begin{proof}
Let $\hat{L}$ be the closure of $L$. 
Then $\overline{\mu}_{\hat{L}}(J)=0$ for all sequence $J$ of length $\le k$ without repeated indices, and 
$\overline{\mu}_{\hat{L}}(I)=\mu_{L}(I)$. 
The result then follows immediately from Theorem \ref{thm:main2}. 
\end{proof}

We now show how to use the $P_0$ polynomial to distinguish string links up to link-homotopy.  

\begin{corollary}\label{cor:logvanish2}
Two $n$-string links $L$ and $L'$ are link-homotopic if and only if 
they have same linking numbers and 
$f_{B_I}(L\cdot\overline{L'})=0$ for all nonrepeated sequences $I$, where 
$\overline{L'}$ denotes the horizontal mirror image of $L'$ with the orientation reversed.  
\end{corollary}
\begin{proof} 
The string link $\overline{L'}$ is the inverse of $L'$ under concordance, i.e. $L'\cdot\overline{L'}$ is 
concordant to the trivial string link.
Since concordance of string links implies link-homotopy \cite{Gif,Gol}, 
the two string links $L$ and $L'$ are link-homotopic if and only if  
$L\cdot\overline{L'}$ is link-homotopic to $\1_n$. 
(The result of \cite{Gif,Gol} is given for links in $S^3$. However, it still holds for string
links.)
Corollary~\ref{cor:logvanish2} follows from Theorem \ref{thm:stringlink} and the fact that 
a string-link is link-homotopic to the trivial one if and only if all Milnor link-homotopy 
invariants of the link vanish \cite{HL}. 
\end{proof}

For an $n$-string link $L$ and a sequence $I$ of possibly repeating elements of $\{1,,,.,n\}$, 
we can define a nonrepeated sequence $D(I)$ and a string link $D_I(L)$ with $|D(I)|$ components, 
in a strictly similar way as for links in the introduction. 
By combining Corollary~\ref{cor:logvanish2} and \cite[Proposition~3.3]{yasuharaAGT} we have the following. 
\begin{corollary}
Two string links $L$ and $L'$ cannot be distinguished by Milnor invariants if and only if 
they have same linking numbers and $f_{B_{D(I)}}(D_I(L\cdot\overline{L'}))=0$ for all sequences $I$.  
\end{corollary}

\section{Example}\label{sec:ex}

In this last section, we give a simple example illustrating the necessity of our hypothesis in Theorem \ref{thm:main}.  

Consider a link $L$ which is the split union of two positive Hopf links, with components labelled by $1$, $2$ and $3$, $4$ respectively.  Then, for the sequence $I=1324$, $\overline{\mu}_L(I)$ vanishes since $\Delta(I)=1$.  
$$\textrm{\includegraphics{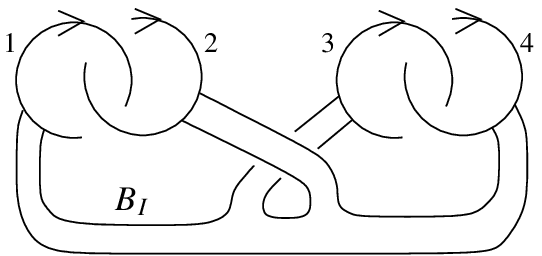}}$$

On the other hand, for the choice of $I$-fusion disk $B_I$ illustrated in the figure, we have 
 $$ 
P_0(L_J) = \left\{
 \begin{array}{ll}
  2t^2-t^4 & \textrm{ if $J=I$, }\\
  1    & \textrm{ if $J\lneq I$. }
 \end{array} \right. $$
Hence we notice that the alternate sum 
$$\sum_{J<I}(-1)^{|J|} (\log P_0(L_J))^{(3)}=\sum_{J<I}(-1)^{|J|} (P_0(L_J))^{(3)}=24$$
is not divisible by $3!2^3=48$.  

This divisibility issue is the main obstruction for our formula to hold in general.

\end{document}

%% file: kn.pstex_t
\begin{picture}(0,0)%
\includegraphics{kn.pstex}%
\end{picture}%
\setlength{\unitlength}{2763sp}%
\begingroup\makeatletter\ifx\SetFigFont\undefined%
\gdef\SetFigFont#1#2#3#4#5{%
  \reset@font\fontsize{#1}{#2pt}%
  \fontfamily{#3}\fontseries{#4}\fontshape{#5}%
  \selectfont}%
\fi\endgroup%
\begin{picture}(8354,1458)(287,-1057)
\put(6035,  5){\makebox(0,0)[lb]{\smash{{\SetFigFont{8}{9.6}{\rmdefault}{\mddefault}{\itdefault}{\color[rgb]{0,0,0}$K_2$}%
}}}}
\put(8295,-467){\makebox(0,0)[lb]{\smash{{\SetFigFont{8}{9.6}{\rmdefault}{\mddefault}{\itdefault}{\color[rgb]{0,0,0}$K_1$}%
}}}}
\put(7374,-457){\makebox(0,0)[lb]{\smash{{\SetFigFont{8}{9.6}{\rmdefault}{\mddefault}{\itdefault}{\color[rgb]{0,0,0}$L_n$}%
}}}}
\put(439,-322){\makebox(0,0)[lb]{\smash{{\SetFigFont{7}{8.4}{\rmdefault}{\mddefault}{\updefault}{\color[rgb]{0,0,0}$\varepsilon_0$}%
}}}}
\put(1837,-344){\makebox(0,0)[lb]{\smash{{\SetFigFont{8}{9.6}{\rmdefault}{\mddefault}{\updefault}{\color[rgb]{0,0,0}$K^\varepsilon_n$}%
}}}}
\put(902,-901){\makebox(0,0)[lb]{\smash{{\SetFigFont{7}{8.4}{\rmdefault}{\mddefault}{\updefault}{\color[rgb]{0,0,0}$\varepsilon_1$}%
}}}}
\put(1519,-906){\makebox(0,0)[lb]{\smash{{\SetFigFont{7}{8.4}{\rmdefault}{\mddefault}{\updefault}{\color[rgb]{0,0,0}$\varepsilon_2$}%
}}}}
\put(2748,-896){\makebox(0,0)[lb]{\smash{{\SetFigFont{7}{8.4}{\rmdefault}{\mddefault}{\updefault}{\color[rgb]{0,0,0}$\varepsilon_n$}%
}}}}
\put(2071,-899){\makebox(0,0)[lb]{\smash{{\SetFigFont{7}{8.4}{\rmdefault}{\mddefault}{\updefault}{\color[rgb]{0,0,0}$\varepsilon_{n-1}$}%
}}}}
\put(3383,-353){\makebox(0,0)[lb]{\smash{{\SetFigFont{7}{8.4}{\rmdefault}{\mddefault}{\updefault}{\color[rgb]{0,0,0}$\varepsilon_{n+1}$}%
}}}}
\end{picture}%

%% file: knclasper.pstex_t
\begin{picture}(0,0)%
\includegraphics{knclasper.pstex}%
\end{picture}%
\setlength{\unitlength}{2763sp}%
\begingroup\makeatletter\ifx\SetFigFont\undefined%
\gdef\SetFigFont#1#2#3#4#5{%
  \reset@font\fontsize{#1}{#2pt}%
  \fontfamily{#3}\fontseries{#4}\fontshape{#5}%
  \selectfont}%
\fi\endgroup%
\begin{picture}(5377,1133)(230,-1093)
\put(2360,-755){\makebox(0,0)[lb]{\smash{{\SetFigFont{7}{8.4}{\rmdefault}{\mddefault}{\updefault}{\color[rgb]{0,0,0}$\varepsilon_n$}%
}}}}
\put(870,-762){\makebox(0,0)[lb]{\smash{{\SetFigFont{7}{8.4}{\rmdefault}{\mddefault}{\updefault}{\color[rgb]{0,0,0}$\varepsilon_1$}%
}}}}
\put(1487,-767){\makebox(0,0)[lb]{\smash{{\SetFigFont{7}{8.4}{\rmdefault}{\mddefault}{\updefault}{\color[rgb]{0,0,0}$\varepsilon_2$}%
}}}}
\put(510,-499){\makebox(0,0)[lb]{\smash{{\SetFigFont{7}{8.4}{\rmdefault}{\mddefault}{\updefault}{\color[rgb]{0,0,0}$\varepsilon_0$}%
}}}}
\put(2741,-491){\makebox(0,0)[lb]{\smash{{\SetFigFont{7}{8.4}{\rmdefault}{\mddefault}{\updefault}{\color[rgb]{0,0,0}$\varepsilon_{n+1}$}%
}}}}
\put(1940,-618){\makebox(0,0)[lb]{\smash{{\SetFigFont{7}{8.4}{\rmdefault}{\mddefault}{\updefault}{\color[rgb]{0,0,0}...}%
}}}}
\end{picture}%

%% file: kane.pstex_t
\begin{picture}(0,0)%
\includegraphics{kane.pstex}%
\end{picture}%
\setlength{\unitlength}{2763sp}%
\begingroup\makeatletter\ifx\SetFigFont\undefined%
\gdef\SetFigFont#1#2#3#4#5{%
  \reset@font\fontsize{#1}{#2pt}%
  \fontfamily{#3}\fontseries{#4}\fontshape{#5}%
  \selectfont}%
\fi\endgroup%
\begin{picture}(7300,1346)(375,-1001)
\put(1786,-356){\makebox(0,0)[lb]{\smash{{\SetFigFont{8}{9.6}{\rmdefault}{\mddefault}{\updefault}{\color[rgb]{0,0,0}$K$}%
}}}}
\put(7266,-26){\makebox(0,0)[lb]{\smash{{\SetFigFont{7}{8.4}{\rmdefault}{\mddefault}{\updefault}{\color[rgb]{0,0,0}$c_1$}%
}}}}
\put(7157,-895){\makebox(0,0)[lb]{\smash{{\SetFigFont{8}{9.6}{\rmdefault}{\mddefault}{\updefault}{\color[rgb]{0,0,0}$K_T$}%
}}}}
\put(6466,176){\makebox(0,0)[lb]{\smash{{\SetFigFont{7}{8.4}{\rmdefault}{\mddefault}{\updefault}{\color[rgb]{0,0,0}$c_{21}$}%
}}}}
\put(4510,-359){\makebox(0,0)[lb]{\smash{{\SetFigFont{7}{8.4}{\rmdefault}{\mddefault}{\updefault}{\color[rgb]{0,0,0}$c_{n1}$}%
}}}}
\put(5607,-153){\makebox(0,0)[lb]{\smash{{\SetFigFont{7}{8.4}{\rmdefault}{\mddefault}{\updefault}{\color[rgb]{0,0,0}$c_{(n-1)2}$}%
}}}}
\put(5165,199){\makebox(0,0)[lb]{\smash{{\SetFigFont{7}{8.4}{\rmdefault}{\mddefault}{\updefault}{\color[rgb]{0,0,0}$c_{(n-1)1}$}%
}}}}
\put(6749,-166){\makebox(0,0)[lb]{\smash{{\SetFigFont{7}{8.4}{\rmdefault}{\mddefault}{\updefault}{\color[rgb]{0,0,0}$c_{22}$}%
}}}}
\put(5035,-422){\makebox(0,0)[lb]{\smash{{\SetFigFont{7}{8.4}{\rmdefault}{\mddefault}{\updefault}{\color[rgb]{0,0,0}$c_{n2}$}%
}}}}
\end{picture}%

%% file: disk.pstex_t
\begin{picture}(0,0)%
\includegraphics{disk.pstex}%
\end{picture}%
\setlength{\unitlength}{4144sp}%
\begingroup\makeatletter\ifx\SetFigFont\undefined%
\gdef\SetFigFont#1#2#3#4#5{%
  \reset@font\fontsize{#1}{#2pt}%
  \fontfamily{#3}\fontseries{#4}\fontshape{#5}%
  \selectfont}%
\fi\endgroup%
\begin{picture}(1499,1684)(583,-1291)
\put(1116,246){\makebox(0,0)[lb]{\smash{{\SetFigFont{10}{12.0}{\rmdefault}{\mddefault}{\updefault}{\color[rgb]{0,0,0}$y_3$}%
}}}}
\put(716, 75){\makebox(0,0)[lb]{\smash{{\SetFigFont{10}{12.0}{\rmdefault}{\mddefault}{\updefault}{\color[rgb]{0,0,0}$y_1$}%
}}}}
\put(907,191){\makebox(0,0)[lb]{\smash{{\SetFigFont{10}{12.0}{\rmdefault}{\mddefault}{\updefault}{\color[rgb]{0,0,0}$y_2$}%
}}}}
\put(1795, 71){\makebox(0,0)[lb]{\smash{{\SetFigFont{10}{12.0}{\rmdefault}{\mddefault}{\updefault}{\color[rgb]{0,0,0}$y_n$}%
}}}}
\put(1215,-1066){\makebox(0,0)[lb]{\smash{{\SetFigFont{10}{12.0}{\rmdefault}{\mddefault}{\updefault}{\color[rgb]{0,0,0}$B_I$}%
}}}}
\put(775,-686){\makebox(0,0)[lb]{\smash{{\SetFigFont{10}{12.0}{\rmdefault}{\mddefault}{\updefault}{\color[rgb]{0,0,0}$x_1$}%
}}}}
\put(954,-686){\makebox(0,0)[lb]{\smash{{\SetFigFont{10}{12.0}{\rmdefault}{\mddefault}{\updefault}{\color[rgb]{0,0,0}$x_2$}%
}}}}
\put(1157,-686){\makebox(0,0)[lb]{\smash{{\SetFigFont{10}{12.0}{\rmdefault}{\mddefault}{\updefault}{\color[rgb]{0,0,0}$x_3$}%
}}}}
\put(1755,-686){\makebox(0,0)[lb]{\smash{{\SetFigFont{10}{12.0}{\rmdefault}{\mddefault}{\updefault}{\color[rgb]{0,0,0}$x_n$}%
}}}}
\end{picture}%